\documentclass[12pt]{article}
\usepackage{amsmath}
\usepackage{amscd}
\usepackage{amsthm}
\usepackage{amssymb} \usepackage{latexsym}
\usepackage{eufrak}
\usepackage{euscript}
\usepackage[width=16.5cm,height=22.5cm]{geometry}
\usepackage{epsfig}
\usepackage{tikz}
\usepackage{graphics}
\usepackage{array}
\usepackage{enumerate}
\usepackage{authblk}
\theoremstyle{theorem}
\newtheorem{theorem}{Theorem}[section]
\theoremstyle{corollary}
\newtheorem{corollary}{Corollary}[section]
\theoremstyle{lemma}
\newtheorem{lemma}{Lemma}[section]
\theoremstyle{definition}

\theoremstyle{proof}

\theoremstyle{remark}
\newtheorem*{rem}{Remark}

\newcommand{\bel}[1]{\begin{equation}\label{#1}}

\newcommand{\be}{\begin{equation}}

\newcommand{\ba}{\begin{eqnarray}}
\newcommand{\ea}{\end{eqnarray}}

\newcommand{\bi}{\bibitem}
\newcommand{\qe}{\end{equation}}

\DeclareFontFamily{OT1}{pzc}{}
\DeclareMathAlphabet{\mathpzc}{OT1}{pzc}{m}{it}

\begin{document}
\setcounter{Maxaffil}{2}

\title{Characteristics polynomial of normalized Laplacian for trees}
\author[1,2]{\rm Anirban Banerjee}
\author[1]{\rm Ranjit Mehatari}
\affil[1]{Department of Mathematics and Statistics}
\affil[2]{Department of Biological Sciences}
\affil[ ]{Indian Institute of Science Education and Research Kolkata}
\affil[ ]{Mohanpur-741246, India}
\affil[ ]{\textit {\{anirban.banerjee, ranjit1224\}@iiserkol.ac.in}}
\date{}
\maketitle
\begin{abstract}
Here, we find the characteristics polynomial of normalized Laplacian of a tree.  The coefficients of this polynomial are expressed by the higher order general Randi\'c  indices for matching, whose values depend on the structure of the tree. We also find the expression of these indices for starlike tree and a double-starlike tree, $H_m(p,q)$. Moreover, we show that two cospectral $H_m(p,q)$ of the same diameter are isomorphic.
\end{abstract}
\textbf{AMS classification: }05C50; 05C05.\\
\textbf{Keywards:}  Normalized Laplacian; Characteristics Polynomial; Tree; Starlike tree; Double Starlike tree; Randi\'c index; Matching; General Randi\'c index for matching; Eigenvalue 1.

\section{Introduction}
Let $\Gamma=(V,E)$ be a simple finite undirected graph of order $n$. Two vertices $u, v\in V$ are called neighbours, $u\sim v$, if they are connected by an edge in $E$, $u\nsim v$ otherwise. Let $d_v$ be the degree of a vertex $v\in V$, that is, the number of neighbours of $v$.
The \textit{normalized Laplacian matrix} \cite{Chung}, $\mathcal{L}$, of $ \Gamma $ is defined as:
\bel{int:eq1}
\mathcal{L}(\Gamma)_{u,v}=\begin{cases}
              1\text{ if } u=v\text{ and }d_v\neq0,\\
              -\frac{1}{\sqrt{d_ud_v}} \text{ if } u\sim v,\\
              0 \text{ otherwise. }
             \end{cases}
\qe
This $\mathcal{L}$ is similar to the normalized Laplacian $\Delta$ defined in \cite{ABJJ,Meh}.
Let  
$ \phi_\Gamma(x)=det(xI-\cal{L})$ be the characteristics polynomial of $\mathcal{L}(\Gamma)$. Let us consider $\phi_\Gamma(x)=a_0x^n-a_1x^{n-1}+a_2x^{n-2}-\cdots+(-1)^{n-1}a_{n-1}x+(-1)^na_n $. Now if $\Gamma$ has no isolated vertices then $a_0=1$, $a_1=n$, $a_2=\frac{n(n-1)}{2}-\sum_{i\sim j}\frac{1}{d_id_j}$ and $a_n=0$ (for some properties of $ \phi_\Gamma(x) $ see \cite{Guo}).
The zeros of $ \phi_\Gamma(x) $ are the eigenvalues of $\cal{L}$ and we order them as 
$\lambda_1\geq\lambda_2\geq\cdots\geq\lambda_{n}=0$. 
 $\Gamma$ is connected \textit{iff} $\lambda_{n-1}>0$. 
 $\lambda_1\le 2$, the equality holds   \textit{iff}  $ \Gamma $ has a bipartite component.
Moreover, $ \Gamma $ is bipartite \textit{iff} for each $ \lambda_i $, the value $ 2-\lambda_i $ is also an eigenvalue of
$ \Gamma $. See \cite{Chung} for more properties of the normalized Laplacian eigenvalues.

\subsection{General Randi\'c index for matching}
There are different topological indices. The degree based topological indices\cite{Gut2,Rad}, which include \textit{Randi\'c index}\cite{Ran}, reciprocal \textit{Randi\'c index}\cite{Li}, general  \textit{Randi\'c index}\cite{BoEd}, higher order \textit{connectivity index}\cite{Ali,Rad1,Rod}, \textit{connective eccentricity index}\cite{Yu1,Yu2}, \textit{Zagreb index}\cite{Abd,Gut1,Gut3,Ham,Kaz,Lin,Vas,Xu} are more popular than others.
\\

For any real number $ \alpha, $ the general Randi\'c index of a graph $\Gamma$ is defined by  B. Bollob\'as and P. Erd\"os (see \cite{BoEd}) as:
\bel{int:eq3}
R_\alpha (\Gamma)=\sum_{i\sim j}(d_id_j)^\alpha,
\qe
which is the general expression of the Randi\'c index  (also known as  connectivity index) introduced by M. Randi\'c in  1975 \cite{Ran} by choosing  $\alpha =- 1/2$ in (\ref{int:eq3}).
For more properties of the general Randi\'c index of graphs, we refer to \cite{Mcav,Li,Li1,Rod,Shi,Shi1}. \\

The Zagreb index of a graph was first introduced by 
Gutman et al.\cite{Gut3} in $1972$. For a graph $\Gamma$ the first and the second Zagreb indices are defined by
\begin{eqnarray*}
Z_1(\Gamma)=\sum_{i\in V}d_i^2\text{ and }
Z_2(\Gamma)=\sum_{i\sim j}d_id_j,
\end{eqnarray*}
 
 respectively. Now, for any positive integer $p $, we define the $p^{th}$ order general \textit{Randi\'c index for matching} as
\bel{}
R_\alpha^{(p)} (\Gamma)
=\sum_{M_p\in \mathcal{M}_p(\Gamma)}\prod_{e\in M_p}s(e)^\alpha,
\qe 
where $s(e)=d_ud_v$ is the \textit{strength} of the edge $e=uv\in  E$, $M_p$ is a $p$-matching, that is, a set of $p$ non-adjacent edges and $\mathcal{M}_p(\Gamma)$ is the set of all $p$-matchings in $\Gamma$. The first order general Randi\'c index for matching with $\alpha=1$ is  the second Zagreb index, that is,  $Z_2(\Gamma)=R_1^{(1)} (\Gamma)$. We take $R_\alpha(\Gamma)=R_\alpha^{(1)}(\Gamma)$ and  $$R_\alpha^{(0)}(\Gamma)=\begin{cases}0\text{ if }\Gamma \text{ is the null graph,}\\
1 \text{ otherwise.}\end{cases}$$ 
If $\Gamma$ is  $r$-regular, then  $R_{\alpha}^{(i)}=r^{2i\alpha}\mid\mathcal{M}_i(\Gamma)\mid$. The   $R_{-1}^{(2)}$ for some known graphs are as follows:
$R_{-1}^{(2)}(S_n)=0$, $R_{-1}^{(2)}(P_n)=\frac{n^2-n-4}{32}$, $R_{-1}^{(2)}(C_n)=\frac{n(n-3)}{32}$, $R_{-1}^{(2)}(K_{p,q})=\frac{(p-1)(q-1)}{4pq}$, and $R_{-1}^{(2)}(K_n)=\frac{3\binom{n}{4}}{(n-1)^4}$, where the notations, $S_n$, $P_n$, $C_n$, $K_n$ and $K_{p,q}$ have their usual meanings.
\begin{theorem}
For any  real number $ \alpha $, $$ 0\leq R_{\alpha}^{(2)}(\Gamma)\leq \frac{1}{2}\Big{[}R_{\alpha}(\Gamma)\Big{]}^2-\frac{1}{2}R_{2\alpha}(\Gamma)$$
\end{theorem}
\begin{proof}
\begin{eqnarray*}
\Big{[}R_{\alpha}(\Gamma)\Big{]}^2&=&\Big{[}\sum_{\substack{e\in E}}{(s(e))^\alpha}\Big{]}^2\\
&=&\sum_{\substack{e\in E}}(s(e))^{2\alpha}+2\sum_{\substack{e_1,e_2\in \mathcal{M}_2(\Gamma)}}(s(e_1)s(e_2))^\alpha+2\sum_{\substack{e_1,e_2\notin \mathcal{M}_2(\Gamma)}}(s(e_1)s(e_2))^\alpha
\end{eqnarray*}
which proves our required result.
\end{proof}
Clearly, for any two graphs $\Gamma_1$, $\Gamma_2$, and $p\ge0$, $R_{\alpha}^{(p)}(\Gamma_1\cup\Gamma_2)\geq R_{\alpha}^{(p)}(\Gamma_1)+R_{\alpha}^{(p)}(\Gamma_2)$. The equality holds, when $p=1$ or one of the graphs is null.

It has been seen that the matching plays a role in the spectrum of a tree. In \cite{chen}, some results on normalized Laplacian spectrum for trees have been discussed.
Now, we derive (or express the coefficients of) the characteristics polynomial $\phi_T(x)$ of a tree $T$  in terms of $R_{-1}^{(i)}(T)$.

\section{The characteristics polynomial of normalized Laplacian for a tree}

\begin{theorem}
\label{cp1}
Let T be a tree with n vertices and maximum matching number k, then 
\begin{equation}
\phi_T(x)=\sum_{i=0}^k(-1)^i(x-1)^{n-2i}R_{-1}^{(i)}(T)
\end{equation}
and the coefficients of $\phi_T$  are given by
\bel{eq1}
a_p=\sum_{i=0}^k(-1)^i\binom{n-2i}{p-2i}R_{-1}^{(i)}(T).
\qe
\end{theorem}
\begin{proof}
Consider a matrix, $B=[b_{ij}]=xI_n-\cal{L}$, where $$b_{ij}=\begin{cases}
x-1\text{ if } i=j\\
\frac{1}{\sqrt{d_id_j}}\text{ if }i\sim j\\
0 \text{ else}.
\end{cases}$$
Now,
\begin{eqnarray*}
\phi_T(x)&=&det(B)\\
&=&\sum_{\sigma\in \mathcal{S}_n } \mathpzc{b}_{\sigma},
 \end{eqnarray*}
 where $\mathpzc{b}_{\sigma} = sgn(\sigma)b_{1,\sigma(1)}b_{2,\sigma(2)}\cdots b_{n,\sigma(n)}$ and  $\mathcal{S}_n$ is the set of all permutation of $\{ 1,\dots,n\}$.
Now,  for any $\sigma\in \mathcal{S}_n$ and $\sigma(i)\neq i$, $\mathpzc{b}_\sigma\neq0$  only when $i\sim \sigma(i)$. 
Since, $T$ does not contain any cycle, here, $\sigma$ is either the identity permutation or a product of disjoint transpositions. 
When $\sigma$ is the identity permutation, $\mathpzc{b}_\sigma=(x-1)^n$ and if  $\sigma=(i_1\sigma(i_1))(i_2\sigma(i_2))\cdots(i_l\sigma(i_l))$ is a product of disjoint transpositions, then, 
\begin{eqnarray*}
\mathpzc{b}_\sigma=\begin{cases}
(-1)^l(x-1)^{n-2l}\frac{1}{d_{i_1}d_{\sigma(i_1)}}\frac{1}{d_{i_2}d_{\sigma(i_2)}}\cdots\frac{1}{d_{i_l}d_{\sigma(i_l)}}\text{ if }i_j\sim\sigma(i_j)\textit{ }\forall j\\
0 \text{ otherwise.}
\end{cases}
\end{eqnarray*}
Now, since, $T$ has the maximum matching number $k$, 
\begin{eqnarray*}
\phi_T(x)&=&\sum_{i=0}^k\sum_{M\in \mathcal{M}_i}(-1)^{\vert M\vert}(x-1)^{n-2\vert M\vert}\prod_{e\in M}\frac{1}{s(e)}\\
&=&\sum_{i=0}^k(-1)^i(x-1)^{n-2i}R_{-1}^{(i)}(T).
\end{eqnarray*}
Expanding the right hand side of the above equation we get 
$$
a_p=\sum_{i=0}^k(-1)^i\binom{n-2i}{p-2i}R_{-1}^{(i)}(T).
$$
\end{proof}

\begin{corollary}
\label{cor1}
For a tree T with maximum matching number k,
\begin{eqnarray}
1-R_{-1}(T)+R_{-1}^{(2)}(T)-\cdots+(-1)^kR_{-1}^{(k)}(T)=0.
\end{eqnarray}
\end{corollary}

\begin{corollary}
\label{cor2}
Let $T$ be a tree with maximum matching number $k$. The eigenvalues of $T$ are $ 1$ with the multiplicity $n-2k$, and  $ 1\pm\sqrt{\alpha_i}$   $(1\leq i\leq k)$ where $ \alpha_i $'s are the zeros of the polynomial 
\begin{eqnarray}
\psi_T(y)= y^{k}-R_{-1}(T)y^{k-1}+\cdots +(-1)^kR_{-1}^{(k)}(T).
\end{eqnarray}
\end{corollary}

The characteristics polynomial $\phi_T(x)$ of a tree $T$  can be expressed in terms of $R_{-1}^{(i)}(T)$, whose value depends on the structure of $T$. Now, we find the expression of  $R_{-1}^{(i)}(T)$ for 
 two different trees, starlike tree \cite{Gut,Rad1} and a specific type of double starlike trees \cite{Laz,Lu}. 

\subsection{Starlike tree}
A tree is called \textit{starlike} tree (see Figure \ref{fig1}) if it has exactly one vertex $ v $ of degree grater than two. 
We denote a starlike tree with $d_v=r$, $3\leq r\leq n-1,$  by $T(l_1,l_2,\ldots ,l_r)$ where $l_i$'s are positive integers with $l_1+l_2+\cdots +l_r=n-1$, that is, 
$T(l_1,l_2,\ldots ,l_r)-v=P_{l_1}\cup P_{l_2}\cup\cdots\cup P_{l_r}$
where $P_{l_i}$ is, a path on $l_i$ vertices, connected to $v$ (see figure (1) for an example). 
Now onwards, without loss of any generality, we assume $1 \le l_1 \le l_2 \le \dots \le l_r$. 
\begin{figure}[h]
\label{fig1}
\centering
\includegraphics[width=12cm]{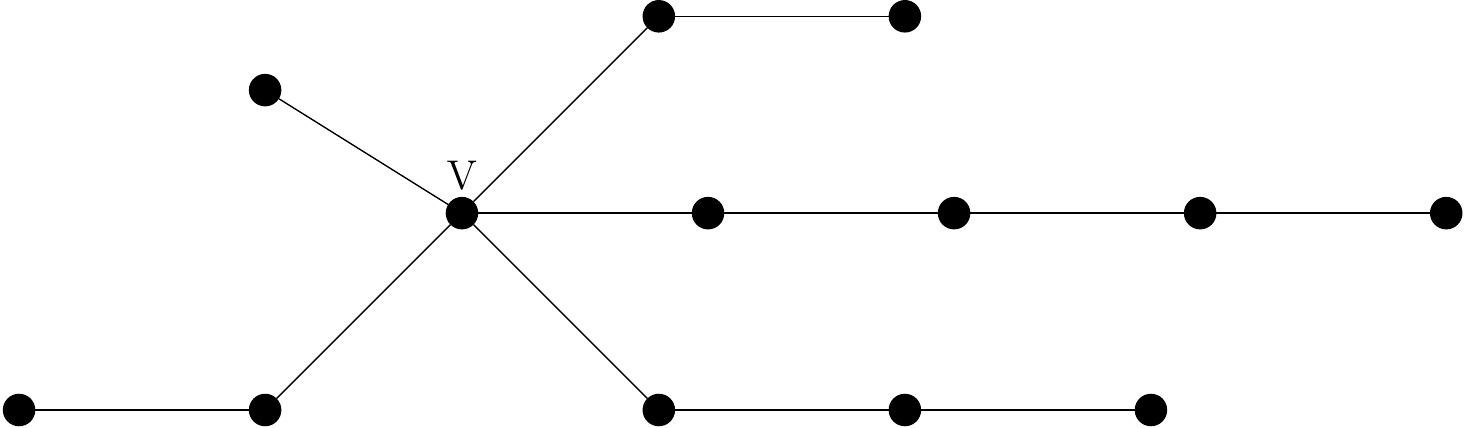}
\caption{ Starlike tree, $T(1,2,2,3,4)$}
\end{figure}
\begin{theorem}
\label{th30}
Let $T(l_1,l_2,\ldots ,l_r)$ be a starlike tree on n vertices with maximum matching number k. If there are exactly m ($m\neq 0$) number of $l_i$'s which are odd numbers then
\begin{enumerate}
\item[(i)]
T has the maximum matching number $ \frac{n-m+1}{2}$, 
\item[(ii)]
 $\displaystyle{R_{-1}(T)=\frac{1}{4}(n+1)+\frac{m}{4}(\frac{2}{r}-1)}$.\\
 
 Furthermore, if $ l_{p_1},l_{p_2},\ldots l_{p_m}$ are  odd  then
\item[(iii)]
$\displaystyle{\vert \mathcal{M}_k(T)\vert=\frac{\prod (l_{p_j}+1)}{2^{m-1}}.\sum_{j=1}^m\frac{1}{l_{p_j}+1}}$, and
\item[(iv)]
$\displaystyle{R_{-1}^{(k)} (T)=\frac{\prod l_{p_j}}{\prod d_i}.\sum_{j=1}^m\frac{1}{l_{p_j}}.}$
\end{enumerate} 
\end{theorem}

\begin{proof}
\begin{enumerate}
\item[(i)]
This part is obvious, since a maximal matching can be taken as follows: for each even $l_i$, we cover the corresponding $P_{l_i}$   by a perfect matching. Consider another  perfect matching on a $P_{l_i+1}$ where $v\in P_{l_i+1}$ and $l_i$ is odd. For all other odd $ l_i $'s, take  $\frac{l_i-1}{2}$ matching.
\item[(ii)] $ T $ has $m$ edges of strength $ r$, $ (r-m) $ edges of strength $ 2r $, $ (r-m) $ edges of strength $ 2 $ and the rest $(n-2r+m-1)$ edges are of strength $ 4 $.
Thus,
\begin{eqnarray*}
R_{-1}(T)&=&\sum_{e\in E}\frac{1}{s(e)}\\
&=&\frac{1}{4}(n+1)+\frac{m}{4}(\frac{2}{r}-1).
\end{eqnarray*} 
\item[(iii)]
$ P_x $ has maximum matching number $ \frac{x-1}{2} $ with $ \frac{x+1}{2} $ number of matchings, when $x$ is odd. Thus, 
\begin{eqnarray*}
\vert \mathcal{M}_k(T)\vert &=&\sum_{j=1}^m\prod_{\substack{k=1\\k\neq j}}^m\frac{(l_{p_k}+1)}{2}\\
&=&\frac{\prod (l_{p_j}+1)}{2^{m-1}}.\sum_{j=1}^m\frac{1}{l_{p_j}+1}.
\end{eqnarray*}
\item[(iv)]
To get  a maximal matching in $T(l_1,l_2,\ldots ,l_r)$, where  $ l_{p_1},l_{p_2}, \ldots, l_{p_m}$ are odd,
 one $l_{p_j}$ is combined with $v_1$ and the   in other odd  $l_{p_j}$ one vertex, of degree one or two,  remains uncovered. 
   One or $ \frac{l_{p_j}-1}{2} $ positions are possible if the degree of the uncover vertex is  one or two respectively. 
Thus, 
\begin{eqnarray*}
R_{-1}^{(k)} (T)&=&\frac{1}{\prod d_i}\Big{[}m+(m-1).2\sum_{j=1}^m\frac{(l_{p_j} -1)}{2}+\cdots +2^{k-1}\sum_{j=1}^m\prod_{\substack{k=1\\k\neq j}}^m\frac{(l_{p_k}-1)}{2}\Big{]}\\
&=&\frac{1}{\prod d_i}\Big{[}\sum_{j=1}^m\prod_{\substack{k=1\\k\neq j}}^m(1+(l_{p_k}-1))\Big{]}\\
&=&\frac{\prod l_{p_j}}{\prod d_i}.\sum_{j=1}^m\frac{1}{l_{p_j}}.
\end{eqnarray*} 
\end{enumerate}
\end{proof}

\begin{corollary}
If $T$ is a tree as in theorem (\ref{th30}), then the multiplicity of the eigenvalue $1$ 
is $m-1$.
\end{corollary}

\begin{rem}
If $ m=0 $ in theorem (\ref{th30}), then $T$ has maximum matching number $ k=\frac{n-1}{2} $ with $ \vert \mathcal{M}_k(T)\vert= \frac{n+1}{2}$ and $R_{-1}^{(k)} (T)=\frac{n-1}{\prod d_i}.$
\end{rem}

\textbf{Example: }  The spectrum of different starlike trees with $ n=8 $ vertices are given bellow. Superscripts in the table show the   algebraic multiplicity of an eigenvalue.
\begin{center}

\begin{tabular}{|l|c|c|r|}
    \hline
    & Partition & Randi\'c  Indices& $\lambda(\cal{L})$ \\
    \hline
   1. & 1,1,5       & $ \frac{25}{12} $, $ \frac{21}{16} $, $ \frac{11}{48} $& 0,2,$ 1^{2} $,$ 1\pm\frac{\sqrt{13\pm\sqrt{37}}}{2\sqrt{6}} $\\
   2. & 1,2,4       &$ \frac{13}{6} $, $ \frac{3}{2} $,$ \frac{17}{48} $,$ \frac{1}{48} $  & 0,2,$1\pm0.876$,$1\pm0.558$,$1\pm0.295$\\
   3. & 1,3,3       & $ \frac{13}{6} $, $ \frac{71}{48} $,$ \frac{5}{16} $ & 0,2,$1^2$,$ 1\pm \frac{\sqrt{3}}{2} $,$ 1\pm \frac{\sqrt{5}}{2\sqrt{3}} $\\
   4. & 2,2,3       & $ \frac{9}{4} $, $ \frac{5}{3} $, $ \frac{7}{16} $,$ \frac{1}{48} $ & 0,2,$1\pm \frac{1}{\sqrt{2}}  $,$ 1\pm\frac{\sqrt{9\pm\sqrt{57}}}{2\sqrt{6}} $\\
   5. & 1,1,1,4     & $ \frac{15}{8} $, $ \frac{31}{32} $, $ \frac{3}{32} $ & 0,2,$ 1^{2} $,$ 1\pm \frac{\sqrt{3}}{2} $,$ 1\pm \frac{1}{2\sqrt{2}} $\\
   6. & 1,1,2,3     & 2,$ \frac{39}{32} $, $ \frac{7}{32} $ & 0,2,$ 1^{2} $,$ 1\pm\frac{\sqrt{4\pm\sqrt{2}}}{2\sqrt{2}} $\\
   7. & 1,2,2,2     & $ \frac{17}{8} $, $ \frac{3}{2} $, $ \frac{13}{32} $, $ \frac{1}{32} $ & 0,2,$ (1\pm \frac{1}{\sqrt{2}})^2 $,$ 1\pm \frac{1}{2\sqrt{2}} $\\
   8. & 1,1,1,1,3   & $ \frac{33}{20} $, $ \frac{13}{20} $& 0,2,$ 1^{4} $,$ 1\pm \sqrt{\frac{13}{20}} $\\
   9. & 1,1,1,2,2   &$ \frac{9}{5} $, $ \frac{19}{20} $, $ \frac{3}{20} $ & 0,2,$ 1^{2} $,$ 1\pm\sqrt{\frac{3}{10}}$,$ 1\pm \frac{1}{\sqrt{2}}$\\
  10. & 1,1,1,1,1,2 & $ \frac{17}{12} $, $ \frac{5}{12} $& 0,2,$ 1^{4} $,$ 1\pm \sqrt{\frac{5}{12}} $\\
    \hline
   \end{tabular}  
\end{center}

\subsection{Double starlike tree}
A tree $T$ is called \textit{double starlike} if it has exactly two vertices of degree greater than two. Let $H_m(p,q)$ be a double starlike tree obtained by attaching $p$ pendant vertices to one end-vertex of a path $P_m$ and $q$ pendant vertices to the other end-vertex of $P_m$. Thus,  $H_2(p,q)$ is a double star $S_{p+1,q+1}$, that is, a tree with exactly two non-pendant vertices with the degree $p+1$ and $q+1$ respectively. See figure (2) for  examples.
\begin{figure}[h]
\label{fig2}
\centering
\includegraphics[width=5.2cm]{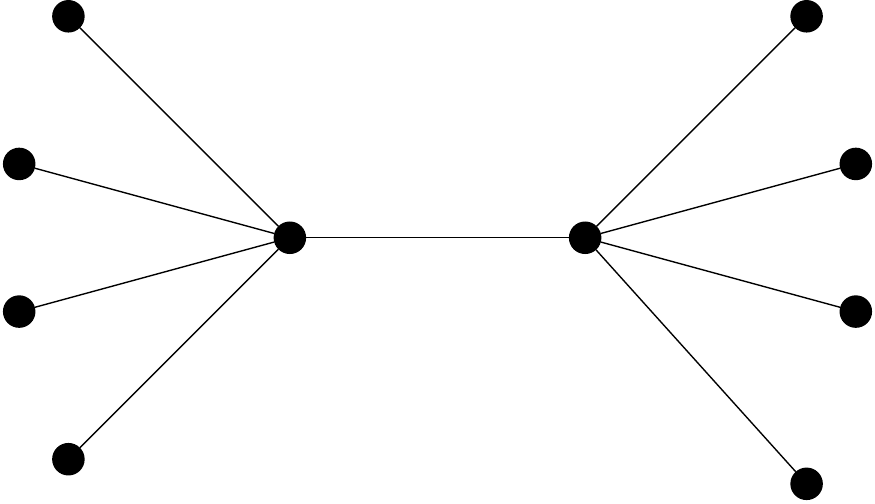}
\hspace{1cm}
\includegraphics[width=7.2cm]{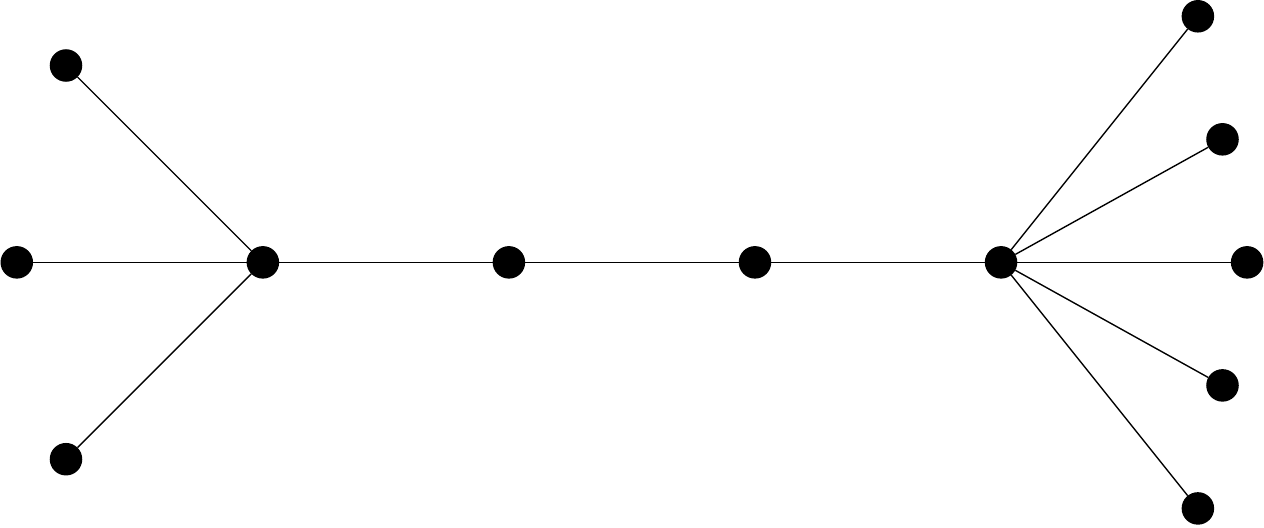}
\caption{ Double starlike tree  $H_2(4,4)$ and $H_4(3,5)$}
\end{figure}
\begin{lemma}
Let $C_{i,n}$ and $k$ be the number of $i$-matchings and maximal matching number, respectively, of $P_n$ ($n\geq3$), then 
$$R_{-1}^{(i)}(P_n)=\frac{1}{4^i}C_{i,n-2}+\frac{1}{4^{i-1}}C_{i-1,n-2}, \text{ and}$$
 $$\psi_{P_n}(y)=(y-1)\sum_{i=0}^{k-1}(-1)^i\frac{1}{4^i}C_{i,n-2}y^{k-1-i}, $$
 where $\psi_{P_n}(y)$ is defined as in corollary (\ref{cor2}).
\end{lemma}
\begin{proof}
 The maximum matching number of $P_n$ is $\lfloor \frac{n}{2}\rfloor $ and  $C_{i,n}=\binom{n-i}{i}$.  $P_n$ has $n-3$ edges of strength 4 and two edges of strength 2. Thus,
\begin{eqnarray*}
R_{-1}^{(i)}(P_n)&=&\frac{1}{4^i} C_{i,n-2}+2\times\frac{1}{2}\times\frac{1}{4^{i-1}} C_{i-1,n-3}+\frac{1}{4^{i-1}}C_{i-2,n-4}\\
&=&\frac{1}{4^i}C_{i,n-2}+\frac{1}{4^{i-1}}C_{i-1,n-2}.
\end{eqnarray*}
which proves the first part of the theorem.\\
For the second part of the lemma, we have, $R^{(k)}_{-1}(P_n)=\frac{1}{4^{k-1}}C_{k-1,n-2}$.\\
Therefore, 
\begin{eqnarray*}
\psi_{P_n}(y)&=&(y^k-y^{k-1})-\frac{1}{4}C_{1,n-2}(y^{k-1}-y^{k-2})+\cdots\\&&+(-1)^{k-1}\frac{1}{4^{k-1}}C_{k-1,n-2}(y-1)\\
&=&(y-1)\sum_{i=0}^{k-1}(-1)^i\frac{1}{4^i}C_{i,n-2}y^{k-1-i}.
\end{eqnarray*}
\end{proof}

\begin{theorem}
\label{th31}
Let $T$ be the double starlike tree $H_m(p,q)$, then
\begin{enumerate}
\item[(i)]
T has the maximum matching number $\lfloor\frac{m}{2}\rfloor+1$, and the multiplicity of the eigenvalue 1 is $p+q-2$ if $m$ is even and  $p+q-1$ if $m$ is odd,
\item[(ii)]
$R_{-1}^{(i)}(T)=R_{-1}^{(i)}(P_m)+\frac{pq}{(p+1)(q+1)}R_{-1}^{(i-1)}(P_m)+\frac{p+q}{2(p+1)(q+1)}R_{-1}^{(i-1)}(P_{m-1})$,
\item[(iii)] and
\begin{eqnarray}
\psi_T(y)=\begin{cases}
(y-r_1)\psi_{P_m}(y)-r_2\psi_{P_{m-1}}(y) &\mbox{if } m \text{ odd,}\\
y(\psi_{P_m}(y)-r_2\psi_{P_{m-1}}(y))-r_1\psi_{P_m}(y)&\mbox{if } m \text{ even,}
\end{cases}
\end{eqnarray}
where $r_1=\frac{pq}{(p+1)(q+1)}$, $r_2=\frac{p+q}{2(p+1)(q+1)}$ and $\psi_T(y)$ is defined as in corollary (\ref{cor2}).
\end{enumerate}
\end{theorem}

\begin{proof}
\begin{enumerate}
\item[(i)] This is easy to verify.
\item[(ii)]
When  $m\geq3$, T has $p$ edges of strength $p+1$, $q$ edges of strength $q+1$, $1$ edge of strength $2(p+1)$, $1$ edges of strength $2(q+1)$ and rest $m-3$ edges of strength $4$. Thus, 
\begin{eqnarray*}
R_{-1}^{(i)}(T)&=&\frac{1}{4^i} C_{i,m-2}
+\frac{1}{4^{i-1}} \Big{[}\frac{p}{p+1}+\frac{q}{q+1}\Big{]}C_{i-1,m-2}+\frac{1}{4^{i-1}}\Big{[}\frac{1}{2(p+1)}+\frac{1}{2(q+1)}\Big{]} C_{i-1,m-3}\\
&&+\frac{1}{4^{i-2}}\frac{pq}{(p+1)(q+1)} C_{i-2,m-2}+\frac{1}{4^{i-2}} \frac{p+q}{2(p+1)(q+1)}C_{i-2,m-3}\\
&&+\frac{1}{4^{i-2}} \frac{1}{4(p+1)(q+1)}C_{i-2,m-4}\\
&=&\frac{1}{4^i} C_{i,m-2}+\frac{1}{4^{i-1}}C_{i-1,m-2}+\frac{1}{4^{i-1}}\frac{pq}{(p+1)(q+1)}C_{i-1,m-2}\\
&&+\frac{1}{4^{i-2}}\frac{pq}{(p+1)(q+1)} C_{i-2,m-2} +\frac{1}{4^{i-1}}\frac{p+q}{2(p+1)(q+1)}C_{i-1,m-3}\\
&&+\frac{1}{4^{i-2}} \frac{p+q}{2(p+1)(q+1)}C_{i-2,m-3}\\
&=&R_{-1}^{(i)}(P_m)+\frac{pq}{(p+1)(q+1)}R_{-1}^{(i-1)}(P_m)+\frac{p+q}{2(p+1)(q+1)}R_{-1}^{(i-1)}(P_{m-1})
\end{eqnarray*}
Again if $m=2$, then
 \begin{eqnarray*}
R_{-1}(T)&=&1+\frac{pq}{(p+1)(q+1)} \text{, and }\\
R_{-1}^{(2)}(T)&=&\frac{pq}{(p+1)(q+1)}.
\end{eqnarray*} 
\item[(iii)]
We have,
$\psi_T(y)=y^{k}-R_{-1}(T)y^{k-1}+\cdots +(-1)^kR_{-1}^{(k)}(T)$ and
$R_{-1}^{(i)}(T)=R_{-1}^{(i)}(P_m)+r_1R_{-1}^{(i-1)}(P_m)+r_2R_{-1}^{(i-1)}(P_{m-1})$, $1\leq i\leq k$, where $k$ is the maximum matching number of $T$. Hence
the maximum matching number of $P_m$ is $k-1$ and the same is of $P_{m-1}$ is $k-1$ if $m$ odd and $k-2$ otherwise.\\
Thus,
\begin{eqnarray}
R_{-1}^{(k)}(T)=\begin{cases}
r_1R_{-1}^{(k-1)}(P_m)+r_2R_{-1}^{(k-1)}(P_{m-1})&\mbox{if } m \text{ odd,}\\
r_1R_{-1}^{(k-1)}(P_m)&\mbox{if } m \text{ even.}
\end{cases}
\end{eqnarray} 
Hence the result follows.
\end{enumerate}
\end{proof}
\textbf{Example: }Consider the double starlike trees as in figure (2). For the first tree, $T_1=H_2(4,4)$ we have $R_{-1}^{(1)}(T_1)=\frac{41}{25}$ and $R_{-1}^{(2)}(T_1)=\frac{16}{25}$. Thus the eigenvalues of $T_1$ are $0$, $1^6$, $2$ $0.2$, $1.8$.
The second tree, $T_2=H_4(3,5)$ has maximum matching number equals to $3$ and the Randi\'c indices of matching are
$R_{-1}^{(1)}(T_2)=\frac{49}{24}$, $R_{-1}^{(2)}(T_2)=\frac{115}{96}$ and $R_{-1}^{(3)}(T_2)=\frac{5}{32}$. Hence the eigenvalues of $T_2$ are $0$, $1^6$, $2$, $1\pm 0.4263$, $1\pm0.9273 $.

\begin{theorem}
Let $T_1=H_m(p_1,q_1)$ and $T_2=H_m(p_2,q_2)$ be $\cal{L}$-cospectral. Then, $T_1$ and $T_2$ are isomorphic.
\end{theorem}
\begin{proof}
From theorem (\ref{th31}) and (\ref{cp1}) we have,  $p_1+q_1=p_2+q_2$ and $p_1q_1=p_2q_2$. Hence the proof.


\end{proof}
\section{Summary and Conclusions}
The Zegreb indeices and  the Randi\'c index are of great importance for molecular chemistry. They are used to characterize the molecular branching in chemical graphs. The general Randi\'c indices for matrching can also play an important role in this  area. It can be used to characterize different classes  of graphs. 
The estimation of general Randi\'c indices for matching for trees, which are the simplest structure amongest the graphs, is much easier than others.
  Corollary  (\ref{cor1}) states that, for a $n$-vertex tree with maximum matching number $k$ it is suffitient to calculate the zeros of a $k$ degree polynomial to determine its complete spectrum. 
  Furthermore, the corollay (\ref{cor1}) shows that the general Randi\'c indices are related by a simple equation. Thus, we only need to calculate the $k-1$ general Randi\'c indices for matching to compute the complete set of eigenvalues.

\section{Acknowledgements}
We are very grateful to the referees for detailed comments and suggestions, which helped to improve the manuscript. The author Ranjit Mehatari is supported by CSIR, Grant no 09/921(0080)/2013-EMR-I, India.

\end{document}